\documentclass[11pt]{article}

\pdftrailerid{}
\usepackage[T1]{fontenc}
\usepackage[utf8]{inputenc}
\usepackage{lmodern,amsmath,amssymb,amsthm,mathtools,mathrsfs}
\usepackage[a4paper,margin=1.02in]{geometry}
\usepackage{booktabs,longtable,array,microtype}
\usepackage{xcolor}
\usepackage[colorlinks=true,linkcolor=blue!50!black,
  citecolor=blue!50!black,urlcolor=blue!50!black]{hyperref}
\usepackage{enumitem,needspace}
\setlist[itemize]{leftmargin=1.5em,itemsep=2pt,topsep=4pt}
\setlist[enumerate]{leftmargin=1.8em,itemsep=2pt,topsep=4pt}
\allowdisplaybreaks
\numberwithin{equation}{section}
\newtheorem{theorem}{Theorem}[section]
\newtheorem{proposition}[theorem]{Proposition}

\newtheorem{corollary}[theorem]{Corollary}
\theoremstyle{definition}

\theoremstyle{remark}

\newcommand{\dd}{\,\mathrm d}

\theoremstyle{plain}
\newtheorem{conjecture}[theorem]{Conjecture}
\newcommand{\Real}{\operatorname{Re}}
\newcommand{\Imag}{\operatorname{Im}}
\newcommand{\De}{D_{5,\mathrm{even}}}
\newcommand{\Do}{D_{5,\mathrm{odd}}}
\hypersetup{
  pdftitle={A quadratic critical-value conjecture for the fifth Bessel moment},
  pdfauthor={Jonas Matuzas},
  pdfsubject={A pure fifth-Bessel conjecture and an exact twisted symmetric-square norm identity},
  pdfkeywords={Bessel moment, modular form, critical L-value, interval arithmetic}}
\title{A quadratic critical-value conjecture\\for the fifth Bessel moment}
\author{Jonas Matuzas\\
\href{mailto:jonas.matuzas@gmail.com}{\nolinkurl{jonas.matuzas@gmail.com}}}
\date{}
\begin{document}
\maketitle
\begin{abstract}
We conjecture a quadratic critical-value evaluation of
$\int_0^\infty K_0(t)^5\,\dd t$ for a weight-three, level-60 newform.
We prove that its twisted symmetric-square value at $2$ equals
$\sqrt{15}\pi^2$ times its Petersson norm, and give exact norm and
coefficient-twist identities. These modular theorems identify the companion
norm formula with Lim--Tu--Yu's conjecture; they do not prove the Bessel
comparison. Chuang's period formulas relate the even and odd determinants.
Directed intervals bound the individual-period discrepancy by $10^{-358}$.
Full proofs and reproducible certificates accompany the note.
\end{abstract}

\noindent\textbf{Generative-AI disclosure and author responsibility.}
The conjecture was obtained with OpenAI's ChatGPT-6-pro, which was also used
for mathematical exploration, drafting, and development of the verification
programs. The author is responsible for the final manuscript.



\section{The individual-period conjecture}
Let $r=\sqrt5>0$ and let $f=\sum a_nq^n$ be the normalized newform in
$S_3(\Gamma_0(60),\chi_{-15})$ with LMFDB label \texttt{60.3.b.a}
\cite{LMFDB}, in the embedding
\[
 a_3=r-2i,\qquad a_5=r(1+2i),\qquad a_7=-8i.
\]
Write $\chi_d$ for the primitive quadratic character of discriminant $d$.
Our classical, unshifted normalization is
\begin{equation}\label{eq:L}
 L(f,s)=\sum_{n\ge1}a_nn^{-s},\qquad
 L:=L(f,2)=4\pi^2\int_0^\infty y f(iy)\,\dd y,
\end{equation}
with analytic continuation of the Dirichlet series.
\begin{conjecture}\label{conj:main}
\begin{equation}\label{eq:main}
 \boxed{A_0:=\int_0^\infty K_0(t)^5\,\dd t
 \stackrel?=\frac\pi4\big[(125-13r)\Real(L^2)+(41r-125)\Imag(L^2)\big].}
\end{equation}
\end{conjecture}
The square precedes the real or imaginary part, and the measure is $\dd t$,
not $t\,\dd t$. Unlike a norm expression involving $|L|^2$, the right side
retains the phase of the critical period. Thus a determinant relation alone
cannot determine this individual entry.

Broadhurst's program \cite{Bro16,Bro17Paris} and the period relations of
Fres\'an--Sabbah--Yu \cite{FSY} and Chuang \cite{Chu} provide the setting.
Lim--Tu--Yu identify $f$ \cite[Theorem~1.2.2(iii)]{LTY}, confirming Evans's
modularity conjecture \cite[\S4]{Evans}, not an individual-period evaluation.
Broadhurst's Paris lecture, reporting joint work with D.~P.~Roberts, already
announces even-power moments, twisted $L$-series and quadratic relations
\cite[pp.~13, 15]{Bro17Paris}. In its notation $A_0=M(0,5,0)$;
its explicit five-Bessel formulas \cite[\S\S2.1, 2.5]{Bro17Paris} and
Zhou's level-15 evaluations \cite{ZhouW} instead use odd powers of $t$.
Our contribution is a specific candidate, not a new Bessel--$L$-value framework.

\section{Two exact modular identities}
For a reproducible definition, put $\eta_d=\eta(d\tau)$ and
\begin{equation}\label{eq:eta}
\begin{gathered}
 P_1=\frac{\eta_2\eta_6^2\eta_{10}^4}{\eta_{30}},\quad
 P_2=\frac{\eta_2^4\eta_{10}\eta_{30}^2}{\eta_6},\quad
 P_3=\frac{\eta_6^4\eta_{10}^2\eta_{30}}{\eta_2},\quad
 P_4=\frac{\eta_2^2\eta_6\eta_{30}^4}{\eta_{10}},\\
 \omega=\frac{1-2i}{r},\quad
 \epsilon=\frac{5+4r+i(2r-10)}{15},\quad
 f=P_1+\omega P_2+3\epsilon P_3+3\epsilon\bar\omega P_4.
\end{gathered}
\end{equation}
The eta criterion and Sturm bound $36$ verify this specification.
The Petersson pairing below is linear in the first variable and is
\emph{not} divided by the index $144$.
\begin{proposition}\label{prop:modular}
Let $\sigma(r)=-r$, $\sigma(i)=-i$, acting on the Fourier coefficients of $f$.
Then
\begin{equation}\label{eq:normtwist}
 \boxed{\langle f,f\rangle_{60}
 =\frac{3(5-r)}{2\pi^4}|L|^2,\qquad
 L(f^\sigma,2)=\frac{(r-1)(4+3i)}{10}L,}
\end{equation}
where $\langle f,f\rangle_{60}
=\int_{\Gamma_0(60)\backslash\mathbb H}|f(\tau)|^2y^3\,\dd x\,\dd y/y^2$.
\end{proposition}
\begin{proof}[Proof sketch]
The ancillary norm proof constructs a holomorphic Eisenstein kernel $H$.
Unfolding with all bad Euler factors retained gives
$\langle f,H\rangle=75\beta|L|^2/(128\pi^4)$,
where $\beta=(1-(r+2i)/9)(1-r(1-2i)/25)$.
A rational Sturm witness and Hecke spectral projection give
$\langle f,H\rangle=\bar\delta\langle f,f\rangle_{60}$,
$\delta=(25+2r+i(5-r))/288$.
The discarded part $(T_7^2+64)R$ is orthogonal to $f$, since
$T_7^*=-T_7$ and $T_7f=-8if$.
Now $75\beta/(128\bar\delta)=3(5-r)/2$ proves the norm formula.

For $w=X\{0,\infty\}$ and
$w_5=\sum_{a=1}^4\chi_5(a)(X-aY/5)\{a/5,\infty\}$, a rational Manin-symbol
witness gives $w_5=w(51-13U_3+3U_3^2-U_3^3)/30+R_5(T_7^2+64)$.
Pairing with $f$ gives $rL(f\otimes\chi_5,2)=[9-r+i(8-2r)]L/5$.
The level-$300$ Sturm identity
$f\otimes\chi_5=f^\sigma-(-r+2ir)f^\sigma(5\tau)$ retains the factor at $5$
and gives the multiplier in \eqref{eq:normtwist}.
These modular computations \cite{Stein} use no Bessel values.
\end{proof}

In particular, the right side of \eqref{eq:main} also equals
$\frac{5\pi}{2}\Real[(7r-5i)L(f,2)L(f^\sigma,2)]$.
This is an exact rewriting of the candidate, not a Bessel evaluation.

\section{The determinants and the remaining comparisons}
For $j\ge0$, let $A_j,B_j,C_j$ denote respectively the integrals of
$t^{2j}K_0^5$, $t^{2j}I_0K_0^4$, and $t^{2j}I_0^2K_0^3$ on $(0,\infty)$.
Lim--Tu--Yu's notation \cite[eq.~(42)]{LTY} specializes to
\begin{equation}\label{eq:Ddef}
 \Do=B_0,\qquad \De=A_0C_1-A_1C_0.
\end{equation}
Here ``even/odd'' refers to the $I_0$-power parity, not the power of $t$;
both determinants use even powers of $t$.
\begin{corollary}[Of Chuang's period formulas]\label{cor:det}
\begin{equation}\label{eq:Drelation}
 \De=\frac{\pi^2}{2\sqrt{15}}\Do.
\end{equation}
\end{corollary}
\begin{proof}
Let $X$ have rows $(A_j)_{j=0}^2$, $(\pi B_j)_{j=0}^2$,
$(\pi^2C_j)_{j=0}^2$.
Specializing \cite[Propositions~13, 18, 24, 32 and Corollary~39]{Chu}
gives $XJX^t=64\pi^6H_B$, where
\[
 H_B=\begin{pmatrix}1/4&0&1/8\\0&-1/10&0\\1/8&0&3/40\end{pmatrix},
 \quad J_{2,*}=(120,0,0),\quad
 \det X=\frac{2\pi^9}{75\sqrt{15}}.
\]
Indices start at zero. From $X^{-1}=JX^tH_B^{-1}/(64\pi^6)$,
$(X^{-1})_{2,1}=-75B_0/(4\pi^5)$.
The cofactor expression for the same entry is $-\pi^2\De/\det X$.
Equating them proves \eqref{eq:Drelation}; the ancillary notes give
all changes of basis.
\end{proof}

The companion norm conjecture and its equivalent determinant form are
\begin{equation}\label{eq:Dnorm}
 \boxed{\Do\stackrel?=\frac{3\sqrt{15}(5-r)}2|L|^2,
 \qquad \De\stackrel?=\frac{3\pi^2(5-r)}4|L|^2.}
\end{equation}
By Proposition~\ref{prop:modular}, the first is equivalently
$B_0\stackrel?=\sqrt{15}\pi^4\langle f,f\rangle_{60}$.
This is still a Bessel-to-modular comparison, not part of that proposition.

Lim--Tu--Yu already conjecture \cite[eq.~(44)]{LTY}
\begin{equation}\label{eq:LTY}
 \Do\stackrel?=\pi^2\mathscr L(2)=8\sqrt{15}\,\mathscr L(3),
 \qquad \mathscr L(s)=L(\chi_{-4}\operatorname{Sym}^2f,s).
\end{equation}
The modular comparison, unlike the Bessel equality in \eqref{eq:LTY}, is exact.
\begin{proposition}[Twisted symmetric-square norm identity]\label{prop:criticalnorm}
In the full Euler-factor normalization of \cite{LTY},
\begin{equation}\label{eq:factorization}
 \boxed{\mathscr L(2)=\sqrt{15}\pi^2\langle f,f\rangle_{60}
 =\frac{3\sqrt{15}(5-r)}{2\pi^2}|L|^2.}
\end{equation}
Moreover, $\mathscr L(3)=\pi^4\langle f,f\rangle_{60}/8$.
\end{proposition}
\begin{proof}[Proof sketch]
Write $\chi=\chi_{-4}$ and
$D(s)=\sum_{m\ge1}\chi(m)a_{m^2}m^{-s}$.
If $\mathcal P_j$ is the weight-three Poincar\'e series of index $j$, set
$\mathcal K_u=\sum_{m\ge1}\chi(m)m^{1-2u}\mathcal P_{m^2}$, initially for
$\Re u>1/4$. Unfolding gives
$\langle f,\mathcal K_u\rangle=D(3+2u)/(16\pi^2)$ for real $u$.
Poisson summation and the Weber--Schafheitlin integral \cite[\S10.22(iv)]{DLMF}
continue this kernel to $u=0$.
Crucially, for discriminants $\Delta=60v^2$ or $0$, the local factor at $2$
vanishes at $u=0$ and cancels the possible zeta pole; the remaining Fourier
sum is uniformly summable. Thus $\mathcal K_0$ is a cusp form, not a formal
critical sum.

Its rational class-number coefficients and Sturm's bound $36$ give
\[
 \mathcal K_0-H=(T_7^2+64)R,\qquad
 H=\frac{49}{144}P_1-\frac5{48}P_2+\frac5{16}P_3+\frac1{16}P_4,
\]
with an explicit rational eta-quotient witness $R$.
Since $T_7^*=-T_7$, spectral projection yields
$\langle f,\mathcal K_0\rangle=\bar\delta\langle f,f\rangle_{60}$,
where $\delta=(7-ir)(7-i)/576$.
The full Euler-factor identity is
$\zeta^{(30)}(2s-4)D(s)=R_0(s)\mathscr L(s)$, with
$R_0(s)=(1+(1+4ir)3^{-s})(1+(15+20i)5^{-s})$;
$\zeta^{(30)}$ omits $2,3,5$.
At $s=3$, $R_0(3)=16(7+ir)(7+i)/675$ and
$\zeta^{(30)}(2)=8\pi^2/75$, so
\[
 \bar\delta\langle f,f\rangle_{60}
 =\frac{75R_0(3)}{128\pi^4}\mathscr L(3),\qquad
 \frac{128\bar\delta}{75R_0(3)}=\frac18.
\]
The functional equation in \eqref{eq:LTY} and Proposition~\ref{prop:modular}
prove \eqref{eq:factorization}.
\end{proof}
Thus \eqref{eq:Dnorm} is equivalent to Lim--Tu--Yu's conjecture \eqref{eq:LTY};
neither proves \eqref{eq:main}. Their Table~1 gives $c_3=\pi^{-2}\Do$ and
$c_4=\pi^{-2}\De$ for $k=5$, not a formula for $A_0$.

\section{Reproducible numerical evidence}
The modular calculation uses the integer eta coefficients in \eqref{eq:eta}
and Mellin folding at $1/\sqrt{60}$. The separate source calculation uses
triple-Bessel and Domb identities \cite{BBBG}, Zhou's Hilbert transform
\cite[eq.~(3.1)]{Zhou}, and positive-series tail bounds.
The odd moment uses Chuang's pairing \cite{Chu}; the even determinant uses its
four $A,C$ entries. Independent decimal and binary directed-arithmetic
implementations enclose truncation and rounding errors:
\begin{center}
\begin{tabular}{lll}
\toprule
Quantity & Decimal value (truncated) & Absolute discrepancy\\
\midrule
$A_0$ against \eqref{eq:main} & $135.2683025808688375942262796\ldots$ & $<10^{-358}$\\
$\Do$ against \eqref{eq:Dnorm} & $27.2910140704559747633195429\ldots$ & $<10^{-352}$\\
$\De$ against \eqref{eq:Dnorm} & $34.7731307499325909815785970\ldots$ & $<10^{-352}$\\
\bottomrule
\end{tabular}
\end{center}
Every residual interval contains zero. The modular implementations agree
within $10^{-422}$ in each component of $L$.
Run \texttt{python -B anc/verify.py} for a standard-library replay of the
intervals and exact witnesses. Complete proofs are in \texttt{anc/proofs};
rational endpoints and a second exact reconstruction are included. Numerical agreement proves none of the
remaining Bessel-to-modular conjectures.


\begin{thebibliography}{12}
\bibitem{LTY}
Z.~L.~Lim, F.-T.~Tu, and J.-D.~Yu,
\emph{Galois representations and modularity for twisted moments of Kloosterman sums},
\href{https://arxiv.org/abs/2609.13576v1}{arXiv:2609.13576v1} (2026).

\bibitem{Chu}
P.-H.~Chuang,
\emph{On the periods of twisted moments of the Kloosterman connection}
(with an appendix in collaboration with J.-D.~Yu),
Ramanujan J. \textbf{65} (2024), no.~3, 1227--1277;
\href{https://arxiv.org/abs/2306.15216v2}{arXiv:2306.15216v2}.

\bibitem{BBBG}
D.~H.~Bailey, J.~M.~Borwein, D.~Broadhurst, and M.~L.~Glasser,
\emph{Elliptic integral evaluations of Bessel moments and applications},
J. Phys. A \textbf{41} (2008), 205203;
\href{https://arxiv.org/abs/0801.0891}{arXiv:0801.0891}.

\bibitem{Bro16}
D.~Broadhurst,
\emph{Feynman integrals, $L$-series and Kloosterman moments},
Commun. Number Theory Phys. \textbf{10} (2016), no.~3, 527--569;
\href{https://arxiv.org/abs/1604.03057}{arXiv:1604.03057}.

\bibitem{Bro17Paris}
D.~Broadhurst,
\emph{$L$-series from Feynman diagrams with up to 22 loops},
Workshop on Multi-Loop Calculations (methods and applications),
Paris, 7 June 2017.
\href{https://multi-loop-2017.sciencesconf.org/data/program/Broadhurst.pdf}{Lecture slides}.

\bibitem{Evans}
R.~Evans,
\emph{Hypergeometric ${}_3F_2(1/4)$ evaluations over finite fields and Hecke eigenforms},
Proc. Amer. Math. Soc. \textbf{138} (2010), no.~2, 517--531.

\bibitem{FSY}
J.~Fres\'an, C.~Sabbah, and J.-D.~Yu,
\emph{Quadratic relations between Bessel moments},
Algebra Number Theory \textbf{17} (2023), no.~3, 541--602;
\href{https://arxiv.org/abs/2006.02702}{arXiv:2006.02702}.

\bibitem{ZhouW}
Y.~Zhou,
\emph{Wick rotations, Eichler integrals, and multi-loop Feynman diagrams},
Commun. Number Theory Phys. \textbf{12} (2018), no.~1, 127--192;
\href{https://arxiv.org/abs/1706.08308}{arXiv:1706.08308} (2017 preprint).

\bibitem{Zhou}
Y.~Zhou,
\emph{Hilbert transforms and sum rules of Bessel moments},
Ramanujan J. \textbf{48} (2019), 159--172;
\href{https://arxiv.org/abs/1706.01068}{arXiv:1706.01068}.

\bibitem{Stein}
W.~Stein,
\emph{Modular Forms: A Computational Approach},
Graduate Studies in Mathematics \textbf{79}, American Mathematical Society, 2007.

\bibitem{DLMF}
NIST Digital Library of Mathematical Functions,
\href{https://dlmf.nist.gov/10.22}{\S10.22(iv)}, Weber--Schafheitlin integrals;
\href{https://dlmf.nist.gov/15.4}{\S15.4}, hypergeometric special values.

\bibitem{LMFDB}
The LMFDB Collaboration,
\emph{The $L$-functions and Modular Forms Database},
newform orbit \href{https://www.lmfdb.org/ModularForm/GL2/Q/holomorphic/60/3/b/a/}{\texttt{60.3.b.a}}.
\end{thebibliography}
\end{document}